\documentclass[11pt]{article}
\usepackage[dvips]{graphicx}
\usepackage{amsmath,amsfonts,amssymb,amsthm,color}

\usepackage[numbers]{natbib}

\newtheorem{theorem}{Theorem}
\newtheorem{lemma}{Lemma}
\newtheorem{corollary}{Corollary}

\newtheorem{proposition}{Proposition}

\newtheorem{definition}{Definition}

\usepackage{bm}
\bmdefine{\Bt}{t}
\bmdefine{\BX}{X}
\bmdefine{\BY}{Y}
\bmdefine{\BZ}{Z}
\bmdefine{\BB}{B}
\bmdefine{\BM}{M}
\bmdefine{\BD}{D}
\bmdefine{\Bi}{i}
\bmdefine{\Bj}{j}
\bmdefine{\Bk}{k}
\bmdefine{\Bx}{x}
\bmdefine{\By}{y}
\bmdefine{\Bz}{z}
\bmdefine{\Bv}{v}
\bmdefine{\Bw}{w}
\bmdefine{\Bn}{n}
\bmdefine{\Ba}{a}
\bmdefine{\Bb}{b}
\bmdefine{\Bc}{c}
\bmdefine{\Be}{e}
\bmdefine{\Bu}{u}
\bmdefine{\Bp}{p}
\bmdefine{\Bzero}{0}
\bmdefine{\Bone}{1}

\newcommand{\N}{{\mathbb N}}

\newcommand{\R}{{\mathbb R}}
\newcommand{\cB}{{\cal B}}

\newcommand{\cI}{{\cal I}}

\newcommand{\cS}{{\cal S}}

\newcommand{\cG}{{\cal G}}
\newcommand{\cN}{{\cal N}}

\newcommand{\figcaption}[1]{\def\cptype{Figure} \caption{#1}}
\newcommand{\tblcaption}[1]{\def\cptype{Table} \caption{#1}}
\newcommand{\figref}[1]{Figure \ref{#1}}

\newcommand{\eq}[1]{(\ref{#1})}

\newcommand{\ip}{i^{\prime}}
\newcommand{\jp}{j^{\prime}}
\newcommand{\idp}{i^{\prime \prime}}
\newcommand{\jdp}{j^{\prime \prime}}

\newcommand{\rsn}{r _1< r _2< \cdots < r _{N+1}}
\newcommand{\csn}{c _1< c _2< \cdots < c _{N+1}}
\newcommand{\cikl}{{\cI _{kl}}}
\newcommand{\cilk}{{\cI _{lk}}}
\newcommand{\cikk}{{\cI _{kk}}}
\newcommand{\cill}{{\cI _{ll}}}

\newcommand{\pij}{p_{ij}}
\newcommand{\xij}{x_{ij}}
\newcommand{\xil}{x_{il}}
\newcommand{\xjk}{x_{jk}}
\newcommand{\xik}{x_{ik}}
\newcommand{\xjl}{x_{jl}}

\newcommand{\xip}{x_{i+}}
\newcommand{\xpj}{x_{+j}}

\newcommand{\xSn}{x_{S_n}}

\newcommand{\yij}{y_{ij}}
\newcommand{\zij}{z_{ij}}
\newcommand{\zipj}{z_{i^{\prime}j}}
\newcommand{\zijp}{z_{ij^{\prime}}}
\newcommand{\zipjp}{z_{i^{\prime}j^{\prime}}}
\newcommand{\zidpjp}{z_{i^{\prime \prime}j^{\prime}}}
\newcommand{\zipjdp}{z_{i^{\prime}j^{\prime \prime}}}
\newcommand{\zip}{z_{i+}}
\newcommand{\zpj}{z_{+j}}

\newcommand{\hmij}{\hat{m}_{ij}}

\newcommand{\normz}{\| \Bz \| _1}

\newcommand{\as}{A_S}
\newcommand{\asns}{A_{S_1,\ldots, S_N}}
\newcommand{\atqs}{A_{T_1,\ldots, T_Q}}
\newcommand{\cFt}{{\cal F}_{\Bt}}
\newcommand{\ma}{{\cal M}_A}
\newcommand{\ms}{{\cal M}_S}

\newcommand{\msns}{{\cal M}_{S_1,\ldots ,S_N}}
\newcommand{\mtqs}{{\cal M}_{T_1,\ldots ,T_Q}}
\newcommand{\masns}{{\cal M}_{A_{S_1,\ldots ,S_N}}}

\newcommand{\rsns}{R_{S_1,\ldots, S_N}}
\newcommand{\isns}{I_{S_1,\ldots, S_N}}

\newcommand{\sumiR}{\sum_{i=1}^{R}}
\newcommand{\sumjC}{\sum_{j=1}^{C}}
\newcommand{\sumnN}{\sum_{n=1}^{N}}
\newcommand{\sumqQ}{\sum_{q=1}^{Q}}
\newcommand{\sumij}{\sum_{(i,j) \in \cI}}

\newcommand{\sumijSn}{\sum_{(i,j) \in {S_n}}}

\title{Markov Bases for Typical Block Effect 
Models of \\ Two-way Contingency Tables}
\author{Mitsunori Ogawa%
\thanks{Department of 
Mathematical Informatics, 
Graduate School of Information Science and Technology, 
University of Tokyo.}~
and Akimichi Takemura\footnotemark[1]\ \thanks{JST CREST}}
\date{December 2011}

\begin{document}
\maketitle

\begin{abstract}
Markov basis for statistical model of contingency tables
gives a useful tool for performing the conditional test of
the model via Markov chain Monte Carlo method.
In this paper we derive explicit forms of Markov bases for
change point models and block diagonal effect models, which are 
typical block-wise effect models of two-way contingency tables,
and perform conditional tests with some real data sets.
\end{abstract}


\section{Introduction}
\label{sec:introduction}

Goodness-of-fit tests for statistical models of
contingency tables are usually performed by
the large sample approximation 
to the null distribution of test statistics.
However, as shown in \citet{haberman-1988},
the large sample approximation
may not be appropriate when the expected frequencies
are not large enough. 
In such cases it is desirable to use a conditional testing procedure.
% For log-linear models
% Fisher's exact test can be performed.
In this paper we discuss
a conditional testing via Markov chain Monte Carlo (MCMC) method with
Markov bases.

\citet{diaconis-sturmfels} showed the equivalence of a Markov basis
and a binomial generator for the toric ideal arising from
a statistical model of discrete exponential families
and developed an algebraic sampling method for conditional
distributions.
Thanks to their algorithm,
once we have a Markov basis for a given statistical model,
we can perform a conditional test for the model via MCMC method.
However, the structure of Markov bases are complicated in general.
Many researchers have studied the structures of Markov bases 
in algebraic statistics 
(e.g. \citet{divide-conquer-2004}, \citet{aoki-takemura-2005jscs}, \citet{Rapallo-2006}, \citet{h-a-t-2010}).

In this paper we derive Markov bases for some statistical models of two-way contingency tables
considering the effect of subtables.
It is a well-known fact that the set of square-free moves
of degree two  (basic moves)
forms the minimal Markov basis for complete independence model
of two-way contingency tables.
On the other hand,  when a subtable effect is added to the model,
the set of  basic moves does not necessarily form a Markov basis.
This problem is called {\it two-way subtable sum problem} (\citet{subtable}).
In the previous researches statistical models with one subtable effect are considered.
We consider some statistical models including several subtable effects.

The organization of this paper is as follows. 
In Section \ref{sec:subsump} we summarize the notations and definitions
on Markov bases and introduce two-way subtable sum problems.
In Section \ref{sec:change} we derive the minimal Markov basis
for the configuration arising from two-way change point model
and discuss the algebraic properties of the toric ideal arising from the change point model.
Section \ref{sec:common} and \ref{sec:general}
give the explicit forms of Markov bases for
the configurations arising from some block diagonal effect models.
In Section \ref{sec:examples} we apply the MCMC method with
our Markov bases to some data sets
and confirm that it works well in practice.
We conclude the paper with some remarks
in Section \ref{sec:remarks}.

\section{Two-way subtable sum problem}
\label{sec:subsump}

In this section we summarize notations and definitions on Markov bases
and give a brief review of two-way subtable sum problems.

\subsection{Preliminaries}
\label{subsec:notations}
% In this subsection we summarize notations and definitions on
% contingency tables and Markov bases.
Let $\N = \{ 0,1,2,\ldots \}$ and
let $\Bx = \{ x_{ij} \} , \xij \in \N , i=1,\ldots, R, j=1,\ldots , C$ be 
an $R \times C$ two-way contingency table
with nonnegative integer entries.
Let $\cI = \{ (i,j) \mid 1 \leq i \leq R, 1 \leq j \leq C \}$ 
be the set of cells.
We order the elements of contingency table $\Bx$ lexicographically
and regard $\Bx$ as a column vector.

Let $A$ be a $T \times | \cI |$ zero-one matrix
where $T$ is a positive integer and $|\cI |= RC$.
We assume that the subspace of $\R ^{|\cI |}$ spanned by
the rows of $A$ contains 
the 
$|\cI |$-dimensional row vector $(1,\ldots ,1)$.
For a given $\Bt \in \N ^T$
the set of contingency tables
\[
	\cFt = \{ \Bx \in \N ^{ | \cI | } \mid A \Bx = \Bt \}
\]
is called the $\Bt$-fiber.
An integer table $\Bz$ with $A \Bz = 0$ is called a move for $A$.
Define the degree of move $\Bz$ as $\| \Bz \| _1 /2 = \sumij |\zij | /2$. 
Let $\ma = \{ \Bz \mid A \Bz =0 \}$ 
denote the set of moves for $A$.
A subset $\cB \subseteq \ma$ is called sign-invariant 
if $\Bz \in \cB$ implies $-\Bz \in \cB$.
Markov basis for $A$ is defined as follows.

%%%%%%%%%%%%%%%%%%%%%%%%%%%%%%%
%
% Definition: Markov basis
%
\begin{definition} \label{def:markov}
A sign-invariant finite set of moves $\cB \subseteq \ma$ is 
a Markov basis for $A$,
if for any $\Bt$ and $\Bx , \By \in \cFt (\Bx \neq \By)$
there exist $U >0, \Bz _{v_1}, \ldots, \Bz_{v_U} \in \cB$
such that
\[
	\By = \Bx + \sum_{s=1}^U \Bz _{v_s} \quad
	and \quad
	\Bx + \sum_{s=1}^u \Bz _{v_s} \in \cFt \ for \ 1 \leq u \leq U.
\]
\end{definition}
In this paper we only consider sign-invariant sets of moves
as Markov bases.
Since the row vector $(1, \ldots , 1)$
is contained in the subspace of $\R ^{|\cI |}$ spanned by
the rows of $A$,
$\sum _{(i, j) \in \cI , \zij >0} \zij = - \sum _{(i, j) \in \cI , \zij <0} \zij$
holds for every move $\Bz \in \ma$.
A move $\Bz$ can be written as $\Bz = \Bz^+ - \Bz^-$
where $\Bz ^+=\{ \max (\zij, 0) \}$ and $\Bz ^-=\{ \max (-\zij, 0) \}$.
If there exists a fiber $\cFt = \{ \Bz^+ , \Bz^- \}$,
we say that $\Bz$ is an indispensable move.

Suppose that $\Bx$ and $\By$ are in the same fiber $\cFt$
and the $l_1$-norm $\| \Bx -\By \| _1= \sumij |\xij - \yij | $
is not equal to zero.
We say that $\| \Bx -\By \| _1$ can be reduced 
by a subset $\cB \subseteq \ma$
if there exist $\tau ^+ \geq 0, \tau ^- \geq 0$,
$\tau ^+ + \tau ^- >0$, and sequences of moves $B_s^+ \in \cB$,
$s=1,\ldots ,\tau ^+$, and $B_s^- \in \cB$ ,
$s=1,\ldots ,\tau ^-$, satisfying
\begin{eqnarray*} 
	&& \left\| \Bx - \By + \sum_{s=1}^{\tau ^+} B_s^+ + \sum_{s=1}^{\tau ^-} B_s^- \right\| _1 < \| \Bx - \By \| _1, \\
	&& \Bx + \sum_{s=1}^{\tau ^{\prime}} B_s^+ \in \cFt, \qquad \tau ^{\prime} = 1,\ldots , \tau^+, \\
	&& \By - \sum_{s=1}^{\tau ^{\prime}} B_s^- \in \cFt, \qquad \tau ^{\prime} = 1,\ldots , \tau^-.
\end{eqnarray*}
It is easy to see that 
a subset $\cB \subseteq \ma$ is a Markov basis for $\ma$
if $\| \Bx -\By \| _1$ can be reduced by $\cB$ 
for all $\Bx$ and $\By$ in every fiber $\cFt$.

\subsection{Two-way subtable sum problem}
\label{subsec:review}
In this subsection we introduce two-way subtable sum problems and 
give a brief review of previous researches.

For a contingency table $\Bx$ 
denote the row sums and column sums of $\Bx$ by
\[
	\xip = \sumjC \xij, \ i=1, \ldots ,R, \quad
	\xpj = \sumiR \xij, \ j=1, \ldots ,C.
\]
Let $S_1, \ldots , S_N$ be subsets of $\cI$ and 
define the subtable sums $\xSn, n=1,\ldots ,N$, by 
\[
	\xSn = \sumijSn \xij.
\]
We summarize the set of row sums, column sums and 
the subtable sums as a column vector
\[
	\Bt = (x_{1+}, \ldots, x_{R+}, x_{+1}, \ldots, x_{+C}, x_{S_1}, \ldots, x_{S_N}) ^{\prime}.
\]
Then, with an appropriate zero-one matrix $\as$,
the relation between $\Bx$ and $\Bt$ is written by
\[
	\asns \Bx = \Bt .
\]
The set of columns of $\asns$ is 
a configuration defining a toric ideal $I_{\asns}$.
For simplicity we call $\asns$ 
the configuration for $S_1, \ldots ,S_N$
and abbreviate the set of moves $\masns$ for $\asns$ as $\msns$.

Consider a square-free move of degree 2 (basic move) of the form
\[
	\begin{matrix}
				& j		& \jp \\
		i		& +1	& -1	\\
		\ip	& -1	& +1
	\end{matrix}
\]
for $i \neq \ip$ and $j \neq \jp$.  
For simplicity we denote this move by $(i, j) (\ip, \jp) - (\ip, j) (i, \jp)$.
Similarly a move of degree $d$ is denoted
by $(i_1, j_1) \cdots (i_d, j_d) - (i_2, j_1) \cdots (i_d, j_{d-1}) (i_1,j_d)$
for appropriate $i_1, \ldots, i_d$ and $j_1, \ldots, j_d$.
\citet{subtable} and \cite{common-diagonal}
discussed  Markov bases for the configuration $\as$
for one subtable $S \subseteq \cI$.
In \cite{subtable} 
it is shown that the set of basic moves in $\ms$
is a Markov basis for $\as$ if and only if $S$ is either 
$2\times 2$ block diagonal or triangular.
\citet{ohsugi-hibi-subtable} discussed the same problem
from algebraic viewpoint.

In Sections \ref{sec:change}--\ref{sec:general}
we use the notations on signs of subtables:
Let $\Bx, \By$ be the two contingency tables in the same fiber $\cFt$
and let $\Bz = \Bx - \By$.
Consider a subset $\hat{\cI} \subseteq \cI$.
If $\zij =0$ for every $(i,j) \in \hat{\cI}$,
we denote it by $\Bz( \hat{\cI} ) =0$.
If $\Bz( \hat{\cI} ) \neq 0$ and $\zij >0$ for $\exists (i,j) \in \hat{\cI}$,
we denote it by $\Bz( \hat{\cI} ) > 0$.
Other notations such as
$\Bz( \hat{\cI} ) \geq 0, \Bz( \hat{\cI} ) \leq 0 $ 
and $\Bz( \hat{\cI}) < 0$ are defined in the same way.

\section{Markov bases for the change point models}
\label{sec:change}

% In this section we discuss Markov bases for
% the configuration arising from the change point model
% of two-way contingency tables
% introduced by \citet{hirotsu-1997}.

In this section we derive Markov bases for
the configuration arising from the change point model
of two-way contingency tables
and discuss its algebraic properties.

\subsection{The unique minimal Markov bases for the change point models}

In this subsection we derive the unique minimal Markov bases for
the change point models.
Two-way change point models with an unknown change point 
is introduced in \citet{hirotsu-1997}.
We consider the two-way change point models 
with several fixed change points.
We call a subtable $S \subseteq \cI$ a rectangle in $\cI$,
if $S$ has a form
\[
	S=\{ (i,j) \mid a_1 \leq i \leq a_2, b_1\leq j \leq b_2 \}
\]
for $1 \leq a_1 < a_2 \leq R$ and $1 \leq b_1 < b_2 \leq C$.
Let $S_n, n=1,\ldots, N$, be rectangles of $\cI$ satisfying
$S_1 \subset S_2 \subset \cdots \subset S_N \subset \cI$.
Then the change point model is defined by
\begin{eqnarray} \label{eqn:change}
	\log \pij = \mu + \alpha _i + \beta _j %+\sumnN \gamma _{S_n} I_{S_n}(i,j) ,
   +\sumnN \gamma_n I_{S_n}(i,j) ,
\end{eqnarray}
where $ I_{S_n}(i,j)=1$ if $(i,j) \in {S_n}$ and $I_{S_n}(i,j) =0$ otherwise.
The sufficient statistic for this model consists 
%the model \eq{eqn:change} consists
of the row sums, the column sums and the sums of frequencies 
in $S_n, n=1,\ldots,N$.

The first main result of this paper is stated as follows.
%%%%%%%%%%%%%%%%%%%%%%%%%%%%%%%%%
%
% Theorem: Markov basis for change point model
%
\begin{theorem} \label{th:markov_change}
Let $S_n, n=1,\ldots, N$, be the rectangles of $\cI$
with $S_1 \subset S_2 \subset \cdots \subset S_N \subset \cI$.
The set of basic moves in $\msns$ is the unique minimal Markov basis
for the configuration $\asns$.
\end{theorem}
\begin{proof}
After an appropriate interchange of rows and columns,
we may assume that $S_n, n=1,\ldots, N$,  are the subsets of $\cI$ defined by
\[
	S_n = \{ (i,j) \mid 1 \leq i \leq r_n, 1\leq j \leq c_n \} , \qquad n =1,\ldots,N,
\]
where each $r_n , c _n, n=1,\ldots ,N$, is a positive integer 
with $1 < r_1\leq \cdots \leq r_N \leq R$ and $1<c_1\leq \cdots \leq c_N \leq C$.
Since $S_1 \subset S_2 \subset \cdots \subset S_N \subset \cI$,
at least one of $r_n < r_{n+1}$ or $c_n < c_{n+1}$ holds for each $n=1,\ldots ,N-1$
and at least one  of $r_N<R$ or $r_N<C$ holds.

Suppose $\Bz(S_1) \neq 0$.
Then $\Bz$ contains both of positive cell and negative cell in $S_1$.
Denote these cells by $(i, j)$ and $(\ip, \jp)$.
If $j=\jp$, letting $(i, \jdp)$ be a negative cell in the $i$-th row,
 $\normz$ can be reduced 
by a basic move $(i,j) (\ip , \jdp) - (\ip , j) (i, \jdp) \in \msns$.
Let us consider the case of $i \neq \ip$ and $j \neq \jp$.
Let $(i, \jdp)$ and $(\idp, j)$ be negative cells in the $i$-th row
and in the $j$-th column, respectively.
Then $\normz$ can be reduced by a sequence of two basic moves
$(i, \jp) (\ip , \jdp) - (\ip , \jp) (i, \jdp)$
and $(i,j) (\idp , \jp) - (\idp , j) (i, \jp)$.
For the case of $r_N <R$ and $c_N <C$,
it can be shown by the same argument that
if $\Bz$ contains both of positive cell and negative cell
in $\hat{\cI } := \{ (i,j) \mid r_N<i\leq R, c_N<j\leq C \}$,
then $\normz$ can be reduced 
by the set of basic moves in $\msns$.
If $r_N=R$ or $c_N=C$ holds, say $r_N=R$, 
$\Bz$ contains both of positive cell and negative cell in $\cI \setminus S_N$.
Then $\normz$ can be reduced 
by the set of basic moves in $\msns$.

Consider the case of $\Bz(S_1)=0$.
We claim that if $\Bz(S_{n-1})=0$ and $\Bz(S_n) \neq 0$ 
for $1<\exists n\leq N$,
then $\normz$ can be reduced by 
the set of basic moves in $\msns$.
If either  $r_{n-1}=r_n$ or $c_{n-1}=c_n$ holds,
we see that $\normz$ can be reduced
by a basic move in $\msns$.
For the case of $r_{n-1}<r_n$ and $c_{n-1}<c_n$,
let $S_n^{12} = \{ (i,j) \mid 1\leq i \leq r_{n-1}, c_{n-1}<j \leq c_n \}$,
$S_n^{21} = \{ (i,j) \mid r_{n-1}<i \leq r_n, 1\leq j \leq c_{n-1} \}$ and
$S_n^{22} = \{ (i,j) \mid r_{n-1}<i \leq r_n, c_{n-1}<j \leq c_n \}$.
If $\Bz$ contains 
both of positive cell and negative cell
in one of $S_n^{12}$, $S_n^{21}$ or $S_n^{22}$,
it can be similarly shown that $\normz$ can be reduced 
by the set of basic moves in $\msns$.
Then we only need to consider
the case of $\Bz(S_n^{kl}) \geq 0$ or $\Bz(S_n^{kl} ) \leq 0$ for each
$(k,l) \in \{ (1,2), (2,1), (2,2) \}$.
Without loss of generality we can assume $S_n^{12} >0$.
Let $(i,j)$ be a positive cell in $S_n^{12}$
and let $(\ip , j)$ be a negative cell in the $j$-th column.
If $(\ip, j) \in S_n^{22}$, using a negative cell $(i, \jp)$ 
in the $i$-th row,
$\normz$ can be reduced 
by a basic move $(i,j) (\ip , \jp) - (\ip , j) (i, \jp) \in \msns$.
Suppose $(\ip, j) \not\in S_n^{22}$.
There exists a negative cell 
$(\idp, \jdp) \in S_n^{21} \cup S_n^{22}$.
Then $\normz$ can be reduced by a sequence of two basic moves
$(\ip, j) (\idp , \jdp) - (\ip , \jdp) (\idp, j)$
and $(i,j) (\ip , \jp) - (\ip , j) (i, \jp)$.
Therefore the claim is proved.

The remaining part is the case that $\Bz(S_N)=0$ and one of 
$\Bz(\hat{\cI}) \geq 0$ or $\Bz(\hat{\cI}) \leq 0$ holds.
If $\Bz(\hat{\cI}) = 0$, 
$\Bz$ contains a nonzero cell in 
$\{ (i,j) \mid r_N<i \leq R, 1 \leq j \leq c_N \}$
or $\{ (i,j) \mid 1 \leq i \leq r_N, c_N< j \leq C \} $.
It is easy to see that $\normz$ can be reduced 
by a basic move in $\msns$.
Suppose $\Bz(\hat{\cI} )> 0$ and let $(i,j)$ be a positive cell in $\hat{\cI}$.
There exist a negative cell $(i, \jp)$ and a positive cell $(\ip ,\jp)$
with $\ip \neq i$ in $\{ (i,j) \mid r_N<i\leq R, 1 \leq j \leq c_N \}$.
Then $\normz$ can be reduced by a basic move
$(i, j) (\ip , \jp) - (\ip , j) (i, \jp)$.

Since every basic move in $\msns$ is indispensable,
the set of basic moves in $\msns$ is the unique minimal Markov basis
for $\asns$  (see \cite{takemura-aoki-2004aism}).
\end{proof}

\subsection{Algebraic properties of the configuration arising from the change point model}
\label{subsec:algebra}

In this subsection we investigate the algebraic properties of the configuration $\asns$
arising from the change point model.

Let $K$ be a field and 
let $K[\{ u_i \} _{1 \leq i \leq R} \cup \{ v_j \} _{1 \leq j \leq C} \cup \{ w_n \} _{1 \leq n \leq N+1 }]$
be a polynomial ring in $R+C+N$ variables over $K$.
We associate each cell $(i,j) \in S_n \setminus S_{n-1}, 1 \leq n \leq N+1$, to a monomial $u_i v_j w_n$
where $S_0=\emptyset$ and $S_{N+1}=\cI$.
Define $\rsns$ as a semigourp ring generated by those monomials.
Let $K[\Bx ] = K[\{ \xij \} _{(i,j) \in \cI }]$
be a polynomial ring in $RC$ variables over $K$.
Define a surjective map $\pi : K[\Bx ] \rightarrow \rsns$
by $\pi (\xij ) = u_i v_j w_n$ for $1 \leq n \leq N+1$.
Define the toric ideal for the change point model 
as the kernel of $\pi $ and denote it by $\isns$.
See \cite{sturmfels1996} and \cite{CLO-2007} for general facts on toric ideals and their Gr\"{o}bner bases.
From Theorem \ref{th:markov_change}
we already know that the toric ideal $\isns$ is generated by
the quadratic binomials corresponding to basic moves in $\msns$.
Furthermore we have the following Theorem \ref{th:algebra}.
%  by a similar argument to \cite{ohsugi-hibi-subtable}.
Although its proof is similar to \cite{ohsugi-hibi-subtable},
we need to use a lexicographic order different from the order used in \cite{ohsugi-hibi-subtable}.
In fact the toric ideal $\isns$ does not have a quadratic Gr\"{o}bner basis with respect
to the lexicographic order used  in \cite{ohsugi-hibi-subtable}
if $N \geq 2$ and there exist $m,n$ such that $2 \leq m<n \leq N+1$ and 
$r_{m-1}<r_m, c_{m-1}<c_m, r_{n-1}<r_n, c_{n-1}<c_n$ where $r_{N+1}=R$ and $c_{N+1}=C$.
\begin{theorem}\label{th:algebra}
For the toric ideal $\isns$
the following statements hold:
\begin{enumerate}
\setlength{\itemsep}{0pt}
\renewcommand{\labelenumi}{(\roman{enumi})}
  \item   $\isns$ possesses a quadratic Gr\"{o}bner basis;
  \item   $\isns$ possesses a square-free initial ideal;
  \item  $\rsns$ is normal;
  \item   $\rsns$ is Koszul.
\end{enumerate}
\end{theorem}
\begin{proof}
Generally, (i) $\Rightarrow$ (iv) and (ii) $\Rightarrow$ (iii) hold.
Since $\rsns$ is generated by the monomials of the same degree,
(i) $\Rightarrow$ (ii) holds from the proof of Proposition 1.6 in \citet{ohsugi-hibi-1999ja}.
Therefore it suffices to show that the statement (i) holds.

By an appropriate interchange of rows and columns,
we may assume that 
$S_n, n=1,\ldots ,N$, share their upper-left corners.
From Theorem \ref{th:markov_change} the toric ideal $\isns$
 is generated by 
\[
	\cG = \{ \xik \xjl - \xil \xjk \mid 1 \leq i < j \leq R, 1 \leq k < l \leq C, 
	\pi ( \xik \xjl ) = \pi ( \xil \xjk ) \}.
\]
Fix a lexicographic order $\succ $ satisfying
$x_{RC} \succ x_{RC-1} \succ \cdots \succ x_{R1} \succ x_{R-1C} \succ \cdots \succ x_{11}$.
Then $\xik \xjl$ is the initial monomial of 
$\xik \xjl - \xil \xjk , 1 \leq i < j \leq R, 1 \leq k < l \leq C$.
We prove that
$\cG$ is a Gr\"{o}bner basis of $\isns$ with respect to $\succ$
using Buchberger's criterion.

Let $f$ be the $S$-polynomial of $g_1, g_2 \in \cG$.
Suppose that $f$ is not reduced to zero by $\cG$.
By Proposition 4 in Section 9 of Chapter 2 in \cite{CLO-2007},
the initial monomials of $g_1$ and $g_2$ are not relatively prime.
On the other hand,  if the monomials of $f$ share a common variable 
$f$ is reduced to zero by $\cG$.
Then $f$ is a cubic binomial and is represented as
$f = x_{i_1l_1} x_{i_2l_2} x_{i_3l_3} - 
x_{i_1^{\prime}l_1^{\prime}} x_{i_2^{\prime}l_2^{\prime}} x_{i_3^{\prime}l_3^{\prime}}$
with the initial monomial $x_{i_1l_1} x_{i_2l_2} x_{i_3l_3}$.
Since $f \in \isns$, we have $\{ i_1, i_2, i_3 \} = \{ i_1^{\prime}, i_2^{\prime}, i_3^{\prime} \}$
and $\{ l_1, l_2, l_3 \} = \{ l_1^{\prime}, l_2^{\prime}, l_3^{\prime} \}$.
Since the monomials of $f$ have no common variable,
 $| \{ i_1, i_2, i_3 \} | = | \{ l_1, l_2, l_3 \} | = 3$.
We assume $1 \leq i_1 = i_1^{\prime} < i_2 = i_2^{\prime} < i_3 = i_3^{\prime} \leq R$
without loss of generality.
By the definition of $\succ$, $l_3 > l_3^{\prime} \in \{ l_1, l_2 \}$.
Then $f$ is represented as one of the following forms:
\begin{enumerate}
\setlength{\itemsep}{0pt}
\renewcommand{\labelenumi}{(\arabic{enumi})}
  \item  $x_{i_1j_1} x_{i_2j_2} x_{i_3j_3} - x_{i_1j_2} x_{i_2j_3} x_{i_3j_1}$,
  \item  $x_{i_1j_1} x_{i_2j_2} x_{i_3j_3} - x_{i_1j_3} x_{i_2j_1} x_{i_3j_2}$,
  \item  $x_{i_1j_1} x_{i_2j_3} x_{i_3j_2} - x_{i_1j_3} x_{i_2j_2} x_{i_3j_1}$,
  \item  $x_{i_1j_2} x_{i_2j_1} x_{i_3j_3} - x_{i_1j_1} x_{i_2j_3} x_{i_3j_2}$,
  \item  $x_{i_1j_2} x_{i_2j_1} x_{i_3j_3} - x_{i_1j_3} x_{i_2j_2} x_{i_3j_1}$,
  \item  $x_{i_1j_3} x_{i_2j_1} x_{i_3j_2} - x_{i_1j_2} x_{i_2j_3} x_{i_3j_1}$,
\end{enumerate}
where 
$1 \leq i_1 < i_2 < i_3 \leq R$ and $1 \leq j_1 < j_2 < j_3 \leq C$.
The candidates (1)--(6) of the form of $f$ are obtained as follows:
Suppose $l_1 < l_2 < l_3$ and $l_3^{\prime} = l_1$.
By $l_2 \neq l_2^{\prime}$ we have $(l_1^{\prime}, l_2^{\prime}) = (l_2, l_3)$.
This implies that $f$ corresponds to the type (1).
The forms (2)--(5) are obtained by the same argument.

For each form of (1)--(6), if there exists a quadratic binomial 
$\hat{f}$ such that $in_{\prec} (\hat{f})$ divides $in_{\prec} (f)$, then
$f$ can be reduced to a cubic binomial whose two monomials
share a common variable.
Hence such $\hat{f}$ does not belong to $\cG$.
We derive a contradiction for each form of $f$. 
Note that $\xik \xjl - \xil\xjk \not\in \cG$, 
$1 \leq i < j \leq R, 1 \leq k < l \leq C$,
 is equivalent
to the existence of $n, 1 \leq n \leq N$, such that 
$(i,k) \in S_n \setminus S_{n-1}$ and $(i,l), (j,k), (j,l) \not\in S_n$.
We refer to this equivalence by ($*$).
\begin{description}
\setlength{\itemsep}{0pt}
\item{(1)} Consider a quadratic binomial 
$\hat{f} = x_{i_1j_1} x_{i_2j_2} - x_{i_1j_2} x_{i_2j_1}$
and let $S_n \setminus S_{n-1}$ be a subtable containing $(i_1, j_1)$.
Since $\hat{f} \not\in \cG$ and ($*$),
$(i_1,j_2), (i_2,j_3), (i_3,j_1) \not\in S_n$.
This contradicts $f \in \isns$. 
\item{(2)} Consider a quadratic binomial 
$\hat{f} = x_{i_1j_1} x_{i_2j_2} - x_{i_1j_2} x_{i_2j_1}$
and let $S_n \setminus S_{n-1}$ be a subtable containing $(i_1, j_1)$.
Since $\hat{f} \not\in \cG$ and ($*$),
$(i_1,j_3), (i_2,j_1), (i_3,j_2) \not\in S_n$.
This contradicts  $f \in \isns$. 
\item{(3)} Consider two quadratic binomials 
$\hat{f}_1 = x_{i_1j_1} x_{i_2j_3} - x_{i_1j_3} x_{i_2j_1}$
and $\hat{f}_2 = x_{i_1j_1} x_{i_3j_2} - x_{i_1j_2} x_{i_3j_1}$.
Let $S_n \setminus S_{n-1}$ be a subtable containing $(i_1, j_1)$.
Since $\hat{f}_1 \not\in \cG$ and ($*$),
$(i_2,j_1) \not\in S_n$.
Since $\hat{f}_2 \not\in \cG$ and ($*$),
$(i_1,j_2) \not\in S_n$.
Then $(i_1,j_3), (i_2,j_2), (i_3,j_1) \not\in S_n$,
which contradicts $f \in \isns$. 
\item{(4)} Consider two quadratic binomials 
$\hat{f}_1 = x_{i_2j_1} x_{i_3j_3} - x_{i_2j_3} x_{i_3j_1}$
and $\hat{f}_2 = x_{i_1j_2} x_{i_3j_3} - x_{i_1j_3} x_{i_3j_2}$.
Let $S_n \setminus S_{n-1}$ be a subtable containing $(i_1, j_1)$.
Since $\hat{f}_1 \not\in \cG$, ($*$) and $f \in \isns$,
$(i_2,j_1) \in S_n \setminus S_{n-1}$.
Similarly $(i_1,j_2) \in S_n \setminus S_{n-1}$ follows from 
$\hat{f}_2 \not\in \cG$, ($*$) and $f \in \isns$.
These contradict $f \in \isns$. 
% Since $\hat{f}_2 \not\in \cG$ and ($*$),
% $(i_1,j_2)$ belongs to the subtable which contains none of
% $(i_1,j_3), (i_3,j_2), (i_3,j_3)$.
% These contradict $f \in \isns$. 
\item{(5)} Consider two quadratic binomials
$\hat{f}_1 = x_{i_2j_1} x_{i_3j_3} - x_{i_2j_3} x_{i_3j_1}$ and 
$\hat{f}_2 = x_{i_1j_2} x_{i_3j_3} - x_{i_1j_3} x_{i_3j_2}$.
Let $S_n \setminus S_{n-1}$ be a subtable containing $(i_2, j_1)$.
Since $\hat{f}_1 \not\in \cG$ and ($*$),
$(i_2,j_3), (i_3,j_1) \not\in S_n$ and $(i_1, j_3) \not\in S_n$.
Since $f \in \cG$, $S_n \setminus S_{n-1}$ contains $(i_2,j_2)$.
Similarly $(i_1,j_2) \in S_n \setminus S_{n-1}$ follows from 
$\hat{f}_2 \not\in \cG$, ($*$) and $f \in \isns$.
This contradicts $f \in \isns$. 
% Consider a quadratic binomial 
% $\hat{f} = x_{i_2j_1} x_{i_3j_3} - x_{i_2j_3} x_{i_3j_1}$
% and let $S_n \setminus S_{n-1}$ be a subtable containing $(i_2, j_1)$.
% Since $\hat{f} \not\in \cG$ and ($*$),
% $(i_2,j_3), (i_3,j_1) \not\in S_n$ and $(i_1, j_3) \not\in S_n$.
% Since $f \in \cG$, $S_n \setminus S_{n-1}$ contains $(i_2,j_2)$.
% Then $S_n \setminus S_{n-1}$ also contains $(i_1, j_2)$,
% which contradicts  $f \in \isns$. 
\item{(6)} Consider a quadratic binomial 
$\hat{f} = x_{i_2j_1} x_{i_3j_2} - x_{i_2j_2} x_{i_3j_1}$
and let $S_n \setminus S_{n-1}$ be a subtable containing $(i_2, j_1)$.
Since $\hat{f} \not\in \cG$ and ($*$),
$(i_1,j_2), (i_2,j_3), (i_3,j_1) \not\in S_n$.
This contradicts $f \in \isns$. 
\end{description}
Therefore $\cG$ is a Gr\"{o}bner basis of $\isns$ with respect to $\succ$.
\end{proof}
% \begin{remark}
% Although Theorem \ref{th:algebra} is proved by a similar argument in \cite{ohsugi-hibi-subtable},
% we use the lexicographic order $\prec $ with
% $x_{RC} \succ x_{RC-1} \succ \cdots \succ x_{R1} \succ x_{R-1C} \succ \cdots \succ x_{11}$.
% In \cite{ohsugi-hibi-subtable} another lexicographic order $\prec $ with
% $x_{R1} \succ x_{R2} \succ \cdots \succ x_{RC} \succ x_{R-11} \succ \cdots \succ x_{1C}$
% is used for the proof of their result.
% \end{remark}

\section{Markov bases for common block diagonal effect models}
\label{sec:common}

In this section we introduce the common diagonal effect model
of two-way contingency tables and derive its Markov basis.

Let $S$ denote the set of cells belonging to the diagonal blocks defined by
\[
	S=\{ (i,j) \mid r _n \leq i < r _{n+1}, 
	c _n \leq j < c _{n+1} , 1\leq \exists n \leq N\} ,
\]
where each $r_n , c _n, n=1,\ldots ,N+1, $ is a non-negative integer with
$1= \rsn =R+1$ and $1=\csn =C+1$.
$S$ is an $N\times N$ block diagonal set in the contingency table.
In the common block diagonal effect model, the cell probabilities
$\{ \pij \}$ are defined by
\begin{eqnarray} \label{eqn:common}
	\log \pij = \mu + \alpha _i + \beta _j +\gamma _S I_S(i,j) .
\end{eqnarray}
In the model \eq{eqn:common}, all cells in diagonal blocks have 
the same parameter $\gamma _S$.
The sufficient statistic for %the model 
\eq{eqn:common} consists
of the row sums, the column sums and the sum of frequencies in $S$.
Note that the model \eq{eqn:common} is a generalization 
of the common diagonal effect model
whose Markov basis is discussed in \citet{common-diagonal}.

Since \citet{subtable} showed that for  $N=2$
the set of basic moves in $\ms$ is a Markov basis for $\as$,
we assume $N \geq 3$ in this section.
In order to describe a Markov basis for $\as$ 
%For the proof of Theorem \ref{th:markov-common}, 
we need some more notations.
% We  consider off-diagonal blocks of $S^C$. 
\begin{figure}[!h]
\begin{center}
\includegraphics[width=4.5cm]{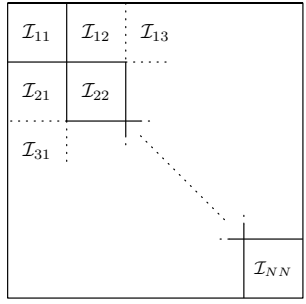}
\vspace{-2mm}
\figcaption{Block-wise indices.}
\label{fig:block-index}
\end{center}
\end{figure}
We index each block as in \figref{fig:block-index},
i.e., 
$\cikl = \{ (i, j) \mid r _k \leq i < r _{k+1}, c _l \leq j < c _{l+1} \}$
for $1 \leq k,l \leq N$.
Note that $S$ can be represented as 
$S=\cI_{11} \cup \cI_{22} \cup \cdots \cup \cI_{NN}$.

%we need to consider 
Consider the following types of moves.  
\begin{itemize}
\setlength{\itemsep}{0pt}
\item Type I (square-free moves of degree 2):
\[
	\begin{matrix}
			& j_1	& j_2 \\
		i_1	& +1	& -1	\\
		i_2	& -1	& +1
	\end{matrix}
\]
where $i_1 \neq i_2$ and $j_1 \neq j_2$.
\item Type II (square-free moves of degree 3):
\[
	\begin{matrix}
			& j_1	& j_2 	& j_3	\\
		i_1	& 0		& +1	& -1	\\
		i_2	& -1	& 0		& +1	\\
		i_3	& +1	& -1	& 0
	\end{matrix}
\]
where nonzero cells (i.e.\ $\pm 1$)  belong to distinct blocks in $S^C$.
\item Type III (square-free moves of degree 3):
\[
	\begin{matrix}
			& j_1	& j_2 	& j_3	\\
		i_1	& +1	& 0		& -1	\\
		i_2	& 0		& -1	& +1	\\
		i_3	& -1	& +1	& 0
	\end{matrix}
\]
where $(i_1, j_1)$ and $(i_2, j_2)$ belong to distinct  blocks in $S$
and other nonzero cells belong to distinct blocks in $S^C$.
\item Type IV (moves of degree 4):
\[
	\begin{matrix}
			& j_1	& j_2 	& j_3	&j_4	\\
		i_1	& +1	& 0		& -1	& 0		\\
		i_2	& 0		& +1	& 0		& -1	\\
		i_3	& 0		& -1	& +1	& 0		\\
		i_4	& -1	& 0		& 0		& +1
	\end{matrix}
\]
where $(i_1, j_1)$ and $(i_3, j_2)$ belong to distinct   blocks in $S$
and other nonzero cells belong to (not necessarily distinct) 
blocks in $S^C$.
%off-diagonal blocks. 
The $i_1$-th and $i_2$-th rows belong to the same block of rows.
Similarly the $i_3$-th and $i_4$-th rows belong to the same block of rows.
There are both square-free and non-square-free moves of this type.
Type IV includes the transpose of these moves.
\end{itemize}

We now give a Markov basis for $\as$ with its explicit form as follows.

%%%%%%%%%%%%%%%%%%%%%%%%%%%%%%%
%
% Theorem: A_S
%
\begin{theorem} \label{th:markov-common}
The set of moves of Types I--VI in $\ms$ forms 
a Markov basis for the configuration $\as$. 
\end{theorem}

We establish Theorem \ref{th:markov-common}
by the lemmas below.
Suppose that $\Bx$ and $\By$ belong to the same fiber $\cFt$
and let $\Bz = \Bx - \By$. 
%
% Let $\Bz = \Bx - \By \neq 0$ denote the difference
% of the two moves $\Bx, \By \in \cFtone$.
The first lemma is proved by the same argument
as Lemma 2 of \cite{subtable} and we omit its proof.
%%%%%%%%%%%%%%%%%%%%%%%%%%%%%%%
%
% Lemma:  + and - in the same block
%
\begin{lemma} \label{lemma:p-and-m}
Suppose that $\Bz$ contains a block $\cikl$
such that there exist two cells $(i,j), (\ip,\jp) \in \cikl$
with $\zij >0$ and $\zipjp <0$.
Then $\normz$ can be reduced by moves of Type I in $\ms$.
\end{lemma}

By this lemma from now on we assume that
every block $\cikl$ in $\Bz$, $1 \leq k,l \leq N$, satisfies
$\Bz(\cikl) \geq 0$ or $\Bz(\cikl) \leq 0$.
Let $K = \{ k \mid 1 \leq k \leq N, \Bz(\cikk) >0 \}$
denote the set of indices of positive diagonal blocks
and let $L = \{ l \mid 1 \leq l \leq N, \Bz(\cill) <0 \}$
denote the set of indices of negative diagonal blocks.

%%%%%%%%%%%%%%%%%%%%%%%%%%%%%%%
%
% Lemma:  S = 0
%
\begin{lemma} \label{lemma:szero}
If $\Bz(S)=0$, then  $\Bz$ can be reduced by a move of Type I or II in $\ms$.
\end{lemma}

\begin{proof}
After an appropriate block-wise interchange of rows and columns,
we assume $\Bz(\cI _{12}) >0$ and $\cI _{11} , \cI_{22} \subseteq S$
without loss of generality.
Let $(i,j)$ be a positive cell in $\cI _{12}$ as shown in \figref{fig:Szero1}.
\begin{figure}[!h]
\begin{center}
\vspace{-2mm}
\includegraphics[]{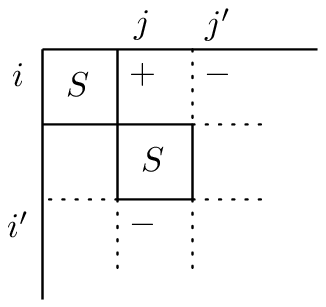}
\vspace{-4mm}
\figcaption{$\Bz$ with $\Bz(S)=0$.}
\label{fig:Szero1}
\vspace{-3mm}
\end{center}
\end{figure}
Then $\Bz$ contains
 negative cells in the $i$-th row and the $j$-th column.
By an appropriate block-wise interchange of rows and columns,
we assume that two of these cells belong to 
$\cI _{13}$ and $\cI _{32}$, respectively.
Note that $\cI _{33}$ may or may not be contained by $S$.
Denote the two negative cells by $(i, \jp)$ and $(\ip ,j)$.
If $\cI _{33} \not\subseteq S$, 
$\normz$ can be reduced by $(i,j) (\ip ,\jp) - (\ip ,j) (i, \jp)$.
Hence let us consider the case of $\cI _{33} \subseteq S$.
Since $z_{+ \jp} =0$,
the $\jp$-th column contains a positive cell.
By an appropriate block-wise interchange of rows,
we assume that this positive cell belongs to $\cI _{23}$ 
as in \figref{fig:Szero2}.
Here, $\cI _{22}$ may or may not be contained by $S$.
Denote the positive cell by $(\idp , \jp)$.
If $\cI _{22} \not\subseteq S$,
$\Bz$ can be reduced by $(i,j) (\idp ,\jp) - (\idp ,j) (i, \jp)$.
If $\cI _{22} \subseteq S$,
there exists a negative cell $(\idp , \jdp) \in \cI _{2 l} , l \neq 2,3$
and $\normz$ can be reduced by
$(i,j) (\ip ,\jdp) (\idp , \jp) - (\ip ,j) (\idp, \jdp) (i, \jp)$.
\end{proof}
\begin{figure}[!h]
\begin{center}
\vspace{-2mm}
\includegraphics[]{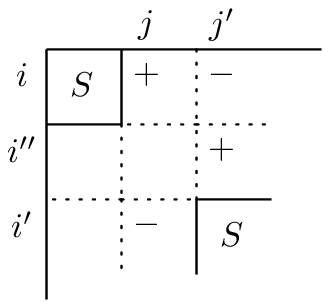}
\vspace{-4mm}
\figcaption{$\Bz$ with $\Bz(S)=0$.}
\label{fig:Szero2}
\vspace{-3mm}
\end{center}
\end{figure}
%%%%%%%%%%%%%%%%%%%%%%%%%%%%%%%
%
% Lemma:  degree 3
%
\begin{lemma} \label{lemma:deg3}
Suppose that $\Bz(\cikk) >0, \Bz(\cill) <0$ and
$\Bz$ contains a cell $(i,j) \in \cikk$ with $\zij >0$ such that
\begin{eqnarray*} 
	r_l \leq \forall \ip < r_{l+1} : \zipj \geq 0 \quad
	\text{and} \quad
	c_l \leq \forall \jp < c_{l+1} : \zijp \geq 0.
\end{eqnarray*}
Then $\normz$ can be reduced by a move of Type III in $\ms$.
\end{lemma}
\begin{proof}
After an appropriate block-wise interchange of rows and columns,
we assume that $k=$ and $l=2$.
Let $(\ip ,\jp)$ be a negative cell in $\cill$.
Since $\zij >0$ and $\zip = \zpj = 0$,
there exist two positive cells $(\idp , j), (i , \jdp)$
with $(\idp , j) \not\in \cikk, \cilk$ 
and $(i, \jdp) \not\in \cikk, \cikl$ as in \figref{fig:Snonzero1}.
Hence, $\normz$ can be reduced by
$(i,j) (\idp ,\jp) (\ip , \jdp) - (\idp ,j) (\ip, \jp) (i, \jdp)$.
\begin{figure}[!h]
\begin{center}
\vspace{-2mm}
\includegraphics[]{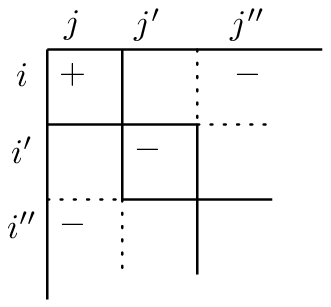}
\vspace{-4mm}
\figcaption{$\Bz$ with the condition of Lemma \ref{lemma:deg3}.}
\label{fig:Snonzero1}
\vspace{-3mm}
\end{center}
\end{figure}
\end{proof}
%%%%%%%%%%%%%%%%%%%%%%%%%%%%%%%
%
% Lemma:  degree 3 or 4
%
\begin{lemma} \label{lemma:deg3or4}
Suppose that $\zij >0, (i,j) \in \cikk$
and there exists $l \in L$ satisfying
\[
	r_l \leq \exists \ip < r_{l+1} : \zipj <0 ~~ \text{and} ~~ \Bz(\cikl) > 0
\]
or
\[
	c_l \leq \exists \jp < c_{l+1} : \zijp <0 ~~ \text{and} ~~ \Bz(\cilk) > 0.
\]
Then $\normz$ can be reduced by a move of Type III or IV in $\ms$.
\end{lemma}
\begin{proof}
After an appropriate block-wise interchange of rows and columns,
we assume that
$k=1, l=2$, $r_2 \leq \exists \ip < r_3: \zipj <0$ and $\Bz(\cI _{12})> 0$.
If there exists a pair of cells 
$(i_1, \jp) \in \cI _{12}$ and $(i_2, \jp) \in \cI _{22}$ 
with $z_{i_1 \jp} > 0, z_{i_2 \jp} < 0$ as in \figref{fig:lem3or4-1},
$\normz$ can be reduced by
$(i,j) (\ip , j_2) (i_1 , \jp) (i_2, j_1) - (\ip ,j) (i_1, j_2) (i_2, \jp) (i, j_1)$.
\begin{figure}[!h]
\begin{center}
\vspace{-2mm}
\includegraphics[]{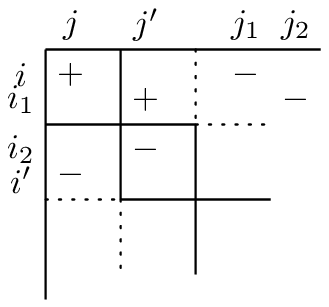}
\vspace{-4mm}
\figcaption{$\Bz$ with the condition of Lemma \ref{lemma:deg3or4}.}
\label{fig:lem3or4-1}
\vspace{-3mm}
\end{center}
\end{figure}
If there exists no such pair,
$\Bz$ satisfies $\Bz(\cI _{11}) >0, \Bz(\cI _{22}) <0$ and
there exists a cell $(\ip , \jp) \in \cI _{22}$ with $- \zipjp >0$
such that $r_1 \leq \forall \idp < r_2: -\zidpjp \geq 0$ 
and $c_1 \leq \forall \jdp < c_2: -\zipjdp \geq 0$.
Hence, by Lemma \ref{lemma:deg3}, 
$\normz$ can be reduced by a move of Type III.
\end{proof}

\begin{proof}[Proof of Theorem \ref{th:markov-common}]
By Lemmas \ref{lemma:p-and-m}, \ref{lemma:szero} and \ref{lemma:deg3or4},
it is enough to show that
$\normz$ can be reduced by a move %low degree move 
of types I--IV
under the following conditions:
\begin{itemize}
\setlength{\itemsep}{0pt}
  \item $\Bz(\cikl) \geq 0$ or $\Bz(\cikl) \leq 0$ holds for all $1 \leq k, l \leq N$ and $\Bz(S) \neq 0$.
  \item For every $k \in K$, $l \in L$ and $(i,j) \in \cikk$ with $\zij >0$,
		\[
			 r_l \leq \forall i^{\prime} < r_{l+1} : z_{i^{\prime}j} \geq 0 \quad 
				\text{or} \quad \Bz(\cI_{kl}) \leq  0 \\
		\]
		and
		\[
			c_l \leq \forall j^{\prime} < c_{l+1} : z_{ij^{\prime}} \geq 0 \quad 
				\text{or} \quad \Bz(\cI_{lk}) \leq  0.
		\]
\end{itemize}
For such $\Bz$ fix $k \in K$ and $l \in L$ and
consider the case that 
the above conditions are satisfied 
by $\Bz(\cikl) \leq 0$ and $\Bz(\cilk) \leq 0$
as in (a) of \figref{fig:proof-thm}. 
In \figref{fig:proof-thm} we assume that $k=1$ and $l=2$ without loss of generality.
\begin{figure}[!h]
\begin{center}
\vspace{-2mm}
\includegraphics[]{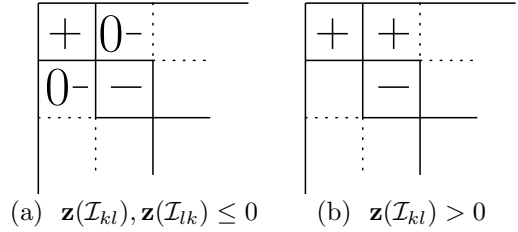}
\vspace{-4mm}
\figcaption{$\Bz$ with the conditions in the proof of Theorem \ref{th:markov-common}.}
\label{fig:proof-thm}
\vspace{-3mm}
\end{center}
\end{figure}
In this case $\normz$ can be reduced by a move of Type III
from the sign-reverse case of Lemma \ref{lemma:deg3}.
Finally, consider the case that at least one of
 $\Bz(\cikl) \leq 0$ and $\Bz(\cilk) \leq 0$
does not hold.
Let $\Bz(\cikl) > 0$ as in (b) of \figref{fig:proof-thm}. 
It is obvious from Lemma \ref{lemma:deg3} 
that $\normz$ can be reduced by a move of Type III.
\end{proof}

\section{Markov bases for general block diagonal effect models}
\label{sec:general}
In the common block diagonal effect model
we assume that every diagonal block has 
the common parameter $\gamma$.
In this section we discuss the case that each diagonal block $S_n$
has its own parameter $\gamma_n$
and more general cases
of block diagonal effect.

We introduce the following model for block diagonal effect 
by the slight modification to \eq{eqn:common}.
Let $S_n, n=1,\ldots ,N$, be the set of cells belonging to the $n$-th 
diagonal block defined as
\[
	S_n=\{ (i,j) \mid r _n \leq i < r _{n+1}, 
	c _n \leq j < c _{n+1} \} .
\]
Then the block diagonal effect model with block-wise 
parameters $\gamma_n, n=1,\ldots ,N,$ is defined by
\begin{eqnarray} \label{eqn:own}
	\log \pij = \mu + \alpha _i + \beta _j +\sumnN \gamma_n I_{S_n}(i,j).
\end{eqnarray}
Note that the model \eq{eqn:own} contains 
the quasi-independence model
considered in \citet{common-diagonal} as a special case.

The sufficient statistic for the model \eq{eqn:own} consists of
the row sums, column sums and the sums of frequencies 
in the block diagonal sets $S_n, n=1,\ldots,N$, and is summarized as
\[
	\Bt = (x_{1+}, \ldots, x_{R+}, x_{+1}, \ldots, x_{+C}, x_{S_1}, \ldots, x_{S_N}) ^{\prime}.
\]
Then a Markov basis for $\asns$ is obtained by essentially the 
same arguments in the proof of Theorem \ref{th:markov-common}.
%%%%%%%%%%%%%%%%%%%%%%%%%%%%%%%
%
% Proposition: A_Sn
%
\begin{proposition} \label{prop:markov-own}
If $N=2$, the set of moves of Type I in $\msns$
forms the unique minimal Markov basis for the configuration $\asns$. 
If $N \geq 3$, the set of moves of Types I and II
in $\msns$ forms the unique minimal Markov basis for the configuration $\asns$. 
\end{proposition}
\begin{proof}
When $N=2$, it is easy to see that
fixing the sums in $\Bt$ is equivalent to 
fixing the sums in $\Bt$ and the sums of $\cI _{12}$ and $\cI _{21}$.
Then, if $\Bz \neq 0$ there exists a block containing 
both a positive cell and a negative cell.
Hence by the same argument in the proof of Lemma \ref{lemma:p-and-m}
$\normz$ can be reduced by a move of Type I.

Let us consider the case of $N \geq 3$.
If there exists a diagonal block $\Bz(S_n) \neq 0$ for $1 \leq n \leq N$,
$S_n$ contains both a positive cell and a negative cell.
Then, by the same argument in the proof of Lemma \ref{lemma:p-and-m},
$\normz$ can be reduced by a move of Type I.
In the case that $\Bz(S_n) =0, n=1,\ldots ,N$, by the same argument
in the proof of Lemma \ref{lemma:szero},
$\normz$ can be reduced by a move of Type II.

The  uniqueness of the minimal Markov basis follows, since 
the basic moves and  the moves of Type  II in $\msns$ are indispensable.
\end{proof}

In Theorem \ref{th:markov-common} in Section \ref{sec:common}
and Proposition \ref{prop:markov-own},
we assumed
$r_{n+1}=R+1$ and $c_{n+1} =C+1$ in the definition of subtables.
In fact, the set of moves of Types I -- IV forms a Markov basis
for the configurations arising from more general block diagonal effect models.
Let $1= \rsn \leq R+1$, $1=\csn \leq C+1$
and $S_n=\{ (i,j) \mid r _n \leq i < r _{n+1}, c _n \leq j < c _{n+1} \} $.
Denote $\cS = \{ S_1,\ldots, S_N \}$.
Let $T_q, q=1,\ldots, Q$, be the subtables of $\cI$ of the form 
$T_q = \bigcup _{n \in \hat{\cN}_q} S_n$
where $\hat{\cN}_q$ is a subset of $\{ 1, 2, \ldots, N \}$
with $\hat{\cN}_q \cap \hat{\cN}_{q^\prime} = \emptyset$ 
for $1 \leq q <q^{\prime} \leq Q$.
% the disjoint union $\dot{\bigcup}_{q=1}^{Q} \hat{\cN}_q = [N]$.
Then the general block diagonal effect model is defined by
\begin{eqnarray} \label{eqn:general}
	\log \pij = \mu + \alpha _i + \beta _j +\sumqQ \gamma_q I_{T_q}(i,j).
\end{eqnarray} 
The sufficient statistic of the model \eq{eqn:general} is summarized as
\[
	(x_{1+}, \ldots, x_{R+}, x_{+1}, \ldots, x_{+C}, x_{T_1}, \ldots, x_{T_Q}) ^{\prime}.
\]
By the same argument of the proofs of Theorem \ref{th:markov-common}
and Proposition \ref{prop:markov-own},
we have the following corollary.
%%%%%%%%%%%%%%%%%%%%%%%%%%%%%%%%%%%%%
%
% Corollary: general block diagonal effect
%
\begin{corollary} \label{corollary:general}
The set of moves of Types I--VI in $\mtqs$ forms 
a Markov basis for the configuration $\atqs$. 
\end{corollary}

\section{Numerical experiments}
\label{sec:examples}

In this section we apply the MCMC method
with the Markov bases derived in the previous sections
for performing conditional tests of some data sets.

The first example is Table \ref{table:school} which shows the relationship
between school and clothing for 1725 children.
This data is from \citet{gilby-1911}.
Each row represents a primary school of the usual county-council type.
The rows are arranged in ascending order of the wealth of the children.
The children are also classified by their clothing and
the columns are arranged in ascending order of the neatness of their clothing.
\begin{table}[htbp]
\begin{center}
\renewcommand{\tabcolsep}{5pt}
{\small
\tblcaption{Relationship between school and clothing.}
\label{table:school}
\begin{tabular}{|l|cccc|c|}
\hline
& I & II & III & IV \& V & Total \\ \hline
No. 1 & 86 & 49 & 10 & 1 & 146 \\
No. 2 & 102 & 116 & 24 & 3 & 245 \\
No. 3 & 25 & 19 & 2 & 0 & 46 \\
No. 4 & 137 & 98 & 33 & 4 & 272 \\
No. 5 & 209 & 222 & 73 & 16 & 520 \\
No. 6 & 65 & 154 & 71 & 27 & 317 \\
No. 7 & 9 & 33 & 1 & 1 & 44 \\
No. 8 & 3 & 60 & 51 & 21 & 135 \\ \hline
Total & 636 & 751 & 265 & 73 & 1725 \\ \hline
\end{tabular}
}
\end{center}
\end{table}
We set the model \eq{eqn:change} with two subtables
$S_1 = \{ (i,1) \mid 1 \leq i \leq 3 \}$,
$S_2 = \{ (i,j) \mid 1 \leq i \leq 5, 1 \leq j \leq 2 \}$
as a null hypothesis.
Starting from the observed data in Table \ref{table:school}
we run a Markov chain of 100,000 tables including 10,000 burn-in steps
and compute the chi-square statistic
for each sampled table.
The histogram of chi-square statistics is shown in \figref{fig:hist-change}.
In the figure the black line shows the asymptotic distribution $\chi _{19}^2$.
Since the observed data is large enough,
the estimated exact distribution is close to $\chi _{19}^2$.
For the observed data in Table \ref{table:school}
the value of chi-square statistic is 154 
and the approximate $p$-value is essentially zero.
Therefore the change point model \eq{eqn:change} is rejected
at the significant level of $5 \%$.
\begin{figure}[!h]
\begin{center}
\vspace{-2mm}
\includegraphics[width=6cm]{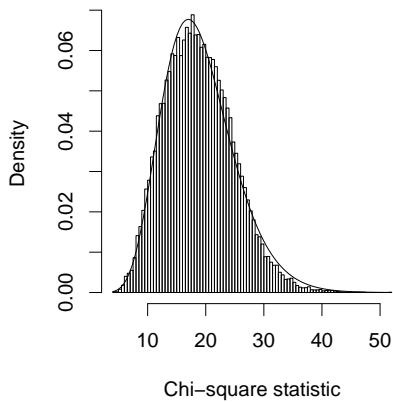}
\vspace{-4mm}
\figcaption{A histogram of chi-square statistic.}
\label{fig:hist-change}
\vspace{-3mm}
\end{center}
\end{figure}

The second example is Table \ref{table:bd} which shows the relationship
between birthday and deathday for 82 descendants of Queen Victoria.
This data is from \citet{diaconis-sturmfels}.
\begin{table}[htbp]
\begin{center}
\renewcommand{\tabcolsep}{5pt}
{\small
\tblcaption{Relationship between birthday and deathday.}
\label{table:bd}
\begin{tabular}{|l|cccccccccccc|}
\hline
& Jan & Feb & March & April & May & June & July & Aug & Sep & Oct & Nov & Dec \\ \hline
Jan & 1 & 0 & 0 & 0 & 1 & 2 & 0 & 0 & 1 & 0 & 1 & 0 \\
Feb & 1 & 0 & 0 & 1 & 0 & 0 & 0 & 0 & 0 & 1 & 0 & 2 \\
March & 1 & 0 & 0 & 0 & 2 & 1 & 0 & 0 & 0 & 0 & 0 & 1 \\
April & 3 & 0 & 2 & 0 & 0 & 0 & 1 & 0 & 1 & 3 & 1 & 1 \\
May & 2 & 1 & 1 & 1 & 1 & 1 & 1 & 1 & 1 & 1 & 1 & 0 \\
June & 2 & 0 & 0 & 0 & 1 & 0 & 0 & 0 & 0 & 0 & 0 & 0 \\
July & 2 & 0 & 2 & 1 & 0 & 0 & 0 & 0 & 1 & 1 & 1 & 2 \\
Aug & 0 & 0 & 0 & 3 & 0 & 0 & 1 & 0 & 0 & 1 & 0 & 2 \\
Sep & 0 & 0 & 0 & 1 & 1 & 0 & 0 & 0 & 0 & 0 & 1 & 0 \\
Oct & 1 & 1 & 0 & 2 & 0 & 0 & 1 & 0 & 0 & 1 & 1 & 0 \\
Nov & 0 & 1 & 1 & 1 & 2 & 0 & 0 & 2 & 0 & 1 & 1 & 0 \\
Dec & 0 & 1 & 1 & 0 & 0 & 0 & 1 & 0 & 0 & 0 & 0 & 0 \\
\hline
\end{tabular}
}
\end{center}
\end{table}
Let $r_i = c_i = 1+3 (i-1)$ for $i=1,\ldots,5$, permitting the replication modulo 12.
We test the common block diagonal effect model \eq{eqn:common} 
against the block diagonal model \eq{eqn:own} with several parameters.
Starting from the observed data in Table \ref{table:bd}
we run a Markov chain of 1,000,000 tables including 100,000 burn-in steps
and compute ($2 \times$) the log-likelihood ratio statistic
\[
		2\sumij \xij \log \frac{\hmij ^2}{\hmij ^1}
\]
for each sampled table $\Bx = \{ \xij \}$ where 
$\hmij ^1$ and $\hmij ^2$ denote
the expected cell frequencies under the model \eq{eqn:common}
and \eq{eqn:own}, respectively.
The histogram of log-likelihood statistics is shown in \figref{fig:hist-llr}.
In the figure the black line shows the asymptotic distribution $\chi _3^2$.
From the sparsity of Table \ref{table:bd}
the estimated exact distribution is different from the asymptotic distribution $\chi _3^2$.
For the observed data in Table \ref{table:bd}
the value of log-likelihood ratio statistic is 3.07 
and the approximate $p$-value is 0.43.
Therefore the common diagonal block effect model \eq{eqn:common} is accepted
at the significant level of $5 \%$.
\begin{figure}[!h]
\begin{center}
\vspace{-14mm}
\includegraphics[width=7cm]{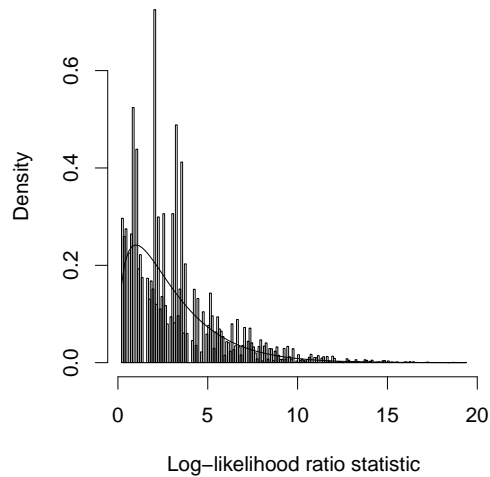}
\vspace{-4mm}
\figcaption{A histogram of log-likelihood ratio statistic.}
\label{fig:hist-llr}
\vspace{-3mm}
\end{center}
\end{figure}

\section{Concluding remarks}
\label{sec:remarks}

In this paper we derive the explicit forms of the Markov bases
for some statistical models with block-wise subtable effects
and perform the conditional testing with some real data sets.
For the change point model we also discussed the algebraic properties 
of the configuration arising from the model.
\citet{subtable} gave the necessary and sufficient condition on the subtable
so that the set of basic moves forms a Markov basis.
It is of interest to consider a necessary and sufficient condition on subtables
$S_1, \ldots ,S_N$ so that the set of basic moves forms a Markov basis.

% \begin{acknowledgements}
% We are very grateful to Hidefumi Ohsugi for valuable discussions.
% We also thank two referees for their valuable and constructive comments.
% \end{acknowledgements}

\section*{Acknowledgment}
We are very grateful to Hisayuki Hara for valuable discussions.

\bibliographystyle{plainnat}
\bibliography{subsum}

\end{document}